\documentclass[12pt]{amsart}

\usepackage{amssymb, amscd, txfonts}
\usepackage{graphicx}

\numberwithin{equation}{section}

\newtheorem{theorem}{Theorem}[section]
\newtheorem{proposition}[theorem]{Proposition}
\newtheorem{lemma}[theorem]{Lemma}

\theoremstyle{definition}

\theoremstyle{remark}
\newtheorem{remark}[theorem]{Remark}

\renewcommand{\hom}{\operatorname{Hom}}
\renewcommand{\ker}{\operatorname{Ker}}

\newcommand{\Z}{\mathbb{Z}}

\newcommand{\C}{\mathbb{C}}

\newcommand{\proj}{{\mathbb P}}

\newcommand{\PGL}{{\rm PGL}}
\newcommand{\SL}{{\rm SL}}
\newcommand{\GL}{{\rm GL}}
\newcommand{\PGLPGL}{{\rm PGL}_2\times{\rm PGL}_2}
\newcommand{\SLSL}{{\rm SL}_2\times{\rm SL}_2}
\newcommand{\Oline}{\mathcal{O}_{{\mathbb P}^{1}}}
\newcommand{\Oplane}{\mathcal{O}_{{\mathbb P}^{2}}}
\newcommand{\OQ}{\mathcal{O}_{Q}}
\newcommand{\OPVb}{\mathcal{O}_{{\mathbb P}V_b}}

\newcommand{\Gab}{{\mathbb G}(a, {\proj}V_b)}
\newcommand{\GaW}{{\mathbb G}(a, {\proj}W)}
\newcommand{\sheaf}{\mathcal{O}}

\DeclareMathOperator{\aut}{Aut}

\begin{document}

\title[]{Rationality of fields of invariants for some representations of ${\SLSL}$}
\author[]{Shouhei Ma}
\address{Graduate~School~of~Mathematics, Nagoya~University, Nagoya 464-8601, Japan}
\email{ma@math.nagoya-u.ac.jp}
\thanks{Supported by Grant-in-Aid for JSPS fellows [21-978] and Grant-in-Aid for Scientific Research (S), No 22224001.} 
\subjclass[2000]{14L30, 14E08, 14H50}
\keywords{${\SLSL}$-representation, rationality problem, rational space curve, transvectant for biform} 
\maketitle 

\begin{abstract}
We prove that the quotient by ${\SLSL}$ of the space of bidegree $(a, b)$ curves on ${\proj}^1\times{\proj}^1$ 
is rational when $ab$ is even and $a\ne b$. 
\end{abstract}

\maketitle


\section{Introduction}\label{sec: intro}

The main objective of this article is to give a simple proof that the fields of invariants are rational 
for some irreducible representations of ${\SLSL}$. 
Such representations are realized as the spaces $V_{a,b}=H^0({\OQ}(a, b))$ of 
biforms of bidegree $(a, b)$ on the surface $Q={\proj}^1\times{\proj}^1$. 
By symmetry we may restrict to the range $a\leq b$. 
In \cite{SB1} Shepherd-Barron proved that ${\proj}V_{3,b}/{\SLSL}$ with $b$ even is rational 
by analyzing transvectants for biforms. 
The case $a=1$, $b$ even $\geq10$ is also settled by him in another paper \cite{SB2}. 
We shall prove the following. 

\begin{theorem}\label{main}
The quotient $|{\OQ}(a, b)|/{\SLSL}$ is rational 
when $a<b$ and $ab$ is even. 
\end{theorem}

Let $V_d$ denote the ${\SL}_2$-representation $H^0({\Oline}(d))$. 
For most $(a, b)$ our proof is based on the following simple idea: 
we identify $V_{a,b}$ with $V_a\otimes V_b={\hom}(V_a^{\vee}, V_b)$, 
and consider the natural fibration 
\begin{equation}\label{eqn: basic map intro}
{\hom}(V_a^{\vee}, V_b) \dashrightarrow {\Gab}
\end{equation}
which associates to a linear map its image in ${\proj}V_b$, 
where ${\Gab}$ is the Grassmannian of $a$-planes in ${\proj}V_b$. 
This is birationally a vector bundle on which 
the first factor of ${\SLSL}$ acts fiberwisely and the second factor acts equivariantly. 
Starting from \eqref{eqn: basic map intro}, we compare several fibrations, and  finally reduce the problem 
to the rationality of ${\proj}V_b/{\SL}_2$ due to Katsylo and Bogomolov \cite{Ka1}, \cite{Ka2}, \cite{B-K}. 

Although we have the fibration \eqref{eqn: basic map intro} for any $a\leq b$, 
there arise difficulties in analyzing it in the following cases: 
\begin{itemize}
\item When $ab$ is odd, a Brauer-Severi scheme over ${\Gab}/{\SL}_2$ becomes birationally nontrivial; 
\item When $a=b$, ${\Gab}$ is one point; 
\item When $a=1$, ${\GL}_2$ acts almost transitively on the fibers of \eqref{eqn: basic map intro}; 
\item For a few other $(a, b)$, ${\PGL}_2$ does not act almost freely on some of relevant spaces. 
\end{itemize}
The first two cases, excluded from Theorem \ref{main}, are the subject of future study. 
For the third case (with $b$ even) we just add a few supplements to the result of \cite{SB2}, mainly using transvectants. 
To study the last case, we identify ${\proj}V_{a,b}$ birationally with 
the space of parametrized rational curves of degree $b$ in ${\proj}^a$. 
We have actually $a=2$ in the relevant cases, and then the rationality is proved by using 
the geometry of rational plane cubics and quartics. 

We note that our argument utilizing the fibration \eqref{eqn: basic map intro} 
will apply more generally to a certain class of representations of product groups. 
In \S \ref{ssec:general} we formulate it in general forms (Propositions \ref{general method I} and \ref{general method II}). 
We then apply it to $V_{a,b}$ in \S \ref{ssec:application}, deducing Theorem \ref{main} for $a>1$, $b>4$.  
In \S \ref{sec: space curve} and \S \ref{sec: transvectant} we treat the remaining few cases in ad hoc ways as above. 


Throughout this article we work over the complex numbers. 
 

\medskip

\noindent
\textbf{Acknowledgement.}
I would like to thank the referee for suggesting to develop \S \ref{ssec:general} in detail, 
which in a previous version was presented only crudely.


\section{Fibration over Grassmannian}\label{sec: main}

In this section we prove Theorem \ref{main} in the main case $a>1$, $b>4$. 
We first explain in \S \ref{ssec:general} the method of proof in a general setting, 
and then apply it in \S \ref{ssec:application} to the present problem.

\subsection{A general method}\label{ssec:general}

Let $V, W$ be representations of algebraic groups $G, H$ respectively. 
We set 
\begin{equation*}
a = {\dim}{\proj}V, \quad b = {\dim}{\proj}W, 
\end{equation*}
and assume that $a\leq b$. 
The tensor product $V\otimes W$ is a representation of $G\times H$. 
We identify $V\otimes W$ with ${\hom}(V^{\vee}, W)$ and 
consider the images of linear maps $V^{\vee}\to W$ that are injective. 
This defines a fibration 
\begin{equation}\label{eqn: basic map}
V\otimes W\dashrightarrow{\GaW} 
\end{equation}
over the Grassmannian ${\GaW}$ of $a$-planes in ${\proj}W$. 
If we denote by $\mathcal{E}\to{\GaW}$ the universal subbundle of rank $a+1$, 
then by \eqref{eqn: basic map} $V\otimes W$ becomes $G\times H$-equivariantly birational to 
the vector bundle $V\otimes\mathcal{E}$ over ${\GaW}$. 
Here $G$ acts on $V$ linearly and $H$ acts on the bundle $\mathcal{E}$ equivariantly. 
Consequently, we have 
\begin{equation}\label{eqn: fibration}
{\proj}(V\otimes W)/G\times H \sim {\proj}(V\otimes\mathcal{E})/G\times H. 
\end{equation}

We shall explain an approach to the rationality problem for ${\proj}(V\otimes W)/G\times H$ utilizing this description. 
Let $G_0\subset G$ (resp. $H_0\subset H$) be the subgroup of elements which act trivially on ${\proj}V$ (resp. ${\proj}W$). 
In particular, $H_0$ acts on the bundle $\mathcal{E}$ by some scalar multiplications. 
We denote $\overline{G}=G/G_0$ and $\overline{H}=H/H_0$. 

\begin{lemma}\label{no-name GaW}
Suppose that (i) $\overline{H}$ acts on ${\GaW}$ almost freely and 
(ii) we have an $H$-linearized line bundle $\mathcal{L}$ over ${\GaW}$ such that 
$H_0$ acts on $\mathcal{E}\otimes\mathcal{L}$ trivially. 
Then  
\begin{equation}\label{eqn:no-name GaW}
{\proj}(V\otimes W)/G\times H \sim ({\proj}V^{\oplus a+1}/G)\times ({\GaW}/H). 
\end{equation}
\end{lemma}

\begin{proof}
By the assumption (ii), the $H$-linearization of the bundle $\mathcal{E}'=\mathcal{E}\otimes\mathcal{L}$
descends to an $\overline{H}$-linearization. 
Then by the assumption (i) we may apply the no-name lemma to $\mathcal{E}'$, 
trivializing it as an $\overline{H}$-linearized vector bundle locally in the Zariski topology. 
Since ${\proj}(V\otimes\mathcal{E})$ is canonically identified with ${\proj}(V\otimes\mathcal{E}')$, 
we obtain the $G\times\overline{H}$-equivariant birational equivalence 
\begin{equation*}
{\proj}(V\otimes\mathcal{E}) = {\proj}(V\otimes\mathcal{E}') \sim 
{\proj}(V\otimes{\C}^{a+1})\times{\GaW}, 
\end{equation*}
where both $G$ and $\overline{H}$ act trivially on the factor ${\C}^{a+1}$. 
\end{proof}

Note that any $H$-linearized line bundle $\mathcal{L}$ over ${\GaW}$ is the tensor product of 
a power of the Pl\"ucker line bundle ${\det}\,\mathcal{E}^{\vee}$ with a 1-dimensional representation of $H$. 

By \eqref{eqn:no-name GaW}, the rationality problem for ${\proj}(V\otimes W)/G\times H$ might be 
decomposed into proving that ${\proj}V^{\oplus a+1}/G$ is rational and that 
${\GaW}/H$ is stably rational of level $\leq {\dim}({\proj}V^{\oplus a+1}/G)$. 
The latter two problems could be studied, for example, via the following reductions. 

\begin{lemma}\label{rationality PVa+1/G}
If $\overline{G}$ acts on ${\proj}V^{\oplus a'}$ almost freely for some $a'\leq a$, we have 
\begin{equation}\label{eqn:rationality PVa+1/G}
{\proj}V^{\oplus a+1}/G \sim {\C}^{(a+1)(a-a'+1)}\times({\proj}V^{\oplus a'}/G). 
\end{equation}
\end{lemma}

\begin{proof}
This is a consequence of the no-name lemma applied to 
the projection ${\proj}V^{\oplus a+1}\dashrightarrow{\proj}V^{\oplus a'}$ from 
some complementary summand $V^{\oplus a-a'+1}$, 
which is a $\overline{G}$-linearized vector bundle. 
\end{proof}

\begin{lemma}\label{stable rationality GaW}
In addition to the assumptions (i), (ii) in Lemma \ref{no-name GaW}, 
suppose furthermore that (iii) $\overline{H}$ acts on ${\proj}W$ almost freely.  
Then we have  
\begin{equation}\label{eqn:stable rationality GaW} 
{\C}^a\times({\GaW}/H) \sim {\C}^{a(b-a)}\times({\proj}W/H). 
\end{equation}
\end{lemma}

\begin{proof}
By the same argument as in the proof of Lemma \ref{no-name GaW}, 
we see that ${\proj}\mathcal{E}/H$ is birational to ${\proj}^a\times({\GaW}/H)$. 
We regard ${\proj}\mathcal{E}$ as the incidence variety 
\begin{equation*}
{\proj}\mathcal{E} = \{ (P, x)\in{\GaW}\times{\proj}W, \; x\in P \} \subset {\GaW}\times{\proj}W. 
\end{equation*}
The fiber of the second projection $\pi\colon{\proj}\mathcal{E}\to{\proj}W$ over $x=[w]\in{\proj}W$ is 
identified with ${\mathbb G}(a-1, {\proj}(W/{\C}w))$. 
Therefore, if $\mathcal{F}\to{\proj}W$ is the universal quotient bundle of rank $b$, 
${\proj}\mathcal{E}$ is identified with the relative Grassmannian ${\mathbb G}(a-1, {\proj}\mathcal{F})$ 
over ${\proj}W$ via $\pi$. 
Then ${\mathbb G}(a-1, {\proj}\mathcal{F})$ is canonically isomorphic to ${\mathbb G}(a-1, {\proj}\mathcal{F}')$ 
for the $H$-linearized bundle $\mathcal{F}'=\mathcal{F}\otimes\mathcal{O}_{{\proj}W}(1)$. 
Since $H_0$ acts on $\mathcal{F}$ and $\mathcal{O}_{{\proj}W}(-1)$ by the same scalars, 
$\mathcal{F}'$ is $\overline{H}$-linearized. 
Now we can use the no-name lemma for $\mathcal{F}'$ to trivialize it 
as an $\overline{H}$-linearized vector bundle locally in the Zariski topology. 
Consequently, we obtain the $\overline{H}$-equivariant birational equivalence  
\begin{equation*}
{\proj}\mathcal{E} \simeq {\mathbb G}(a-1, {\proj}\mathcal{F}') \sim {\mathbb G}(a-1, {\proj}^{b-1})\times{\proj}W, 
\end{equation*}
where $\overline{H}$ acts on the factor ${\mathbb G}(a-1, {\proj}^{b-1})$ trivially. 
\end{proof}

Comparing \eqref{eqn:no-name GaW}, \eqref{eqn:rationality PVa+1/G}, and \eqref{eqn:stable rationality GaW} 
and noticing that $(a+1)(a-a'+1)>a$, we can summarize the above argument in the following proposition. 

\begin{proposition}\label{general method I}
Let $V$, $W$ be representations of $G$, $H$ respectively such that $a={\dim}{\proj}V$ is smaller than $b={\dim}{\proj}W$. 
Assume that 
\begin{enumerate}
\item we have an $H$-linearized line bundle $\mathcal{L}$ as in Lemma \ref{no-name GaW}, 
\item $\overline{G}$ acts on ${\proj}V^{\oplus a'}$ almost freely for some $a'\leq a$, and 
\item $\overline{H}$ acts on ${\proj}W$ and ${\GaW}$ almost freely. 
\end{enumerate}
Then, setting $N=(a+1)(a-a')+1+a(b-a)$, we have 
\begin{equation*}
{\proj}(V\otimes W)/G\times H \sim {\C}^N\times({\proj}V^{\oplus a'}/G)\times({\proj}W/H). 
\end{equation*}
\end{proposition}

In this way, the rationality problem for ${\proj}(V\otimes W)/G\times H$ could be reduced, 
under several hypotheses, 
to results concerning stable rationality of ${\proj}V^{\oplus a'}/G$ and ${\proj}W/H$. 
We would like to mention that for invariant fields of linear representations, 
to prove stable rationality is rather easier than to prove rationality in many cases. 

For our application to ${\SLSL}$-representations, 
we also state a variant deduced from Lemmas \ref{no-name GaW} and \ref{stable rationality GaW}, 
bypassing Lemma \ref{rationality PVa+1/G}. 

\begin{proposition}\label{general method II}
Let $V$, $W$ satisfy the assumptions in Proposition \ref{general method I} except $(2)$. 
Suppose instead that ${\proj}V^{\oplus a+1}/G$ is rational of dimension $d\geq a$. 
Then, setting $M=d-a+a(b-a)$, we have 
\begin{equation*}
{\proj}(V\otimes W)/G\times H \sim {\C}^M\times({\proj}W/H). 
\end{equation*}
\end{proposition}

\begin{remark}
When $a\geq b$, we can instead consider the kernels of linear maps $V^{\vee}\to W$ 
to obtain a fibration $V\otimes W \dashrightarrow {\mathbb G}(a-b-1, {\proj}V^{\vee})$. 
But if we identify ${\mathbb G}(a-b-1, {\proj}V^{\vee})$ with ${\mathbb G}(b, {\proj}V)$ naturally, 
this coincides with the fibration $W\otimes V \dashrightarrow {\mathbb G}(b, {\proj}V)$ 
as in \eqref{eqn: basic map}. 
\end{remark}

\subsection{Application to $V_{a,b}$}\label{ssec:application}

Let $V_d$ denote the ${\SL}_2$-representation $H^0({\Oline}(d))$. 
We shall apply Proposition \ref{general method II} to the ${\SLSL}$-representations $V_{a,b}=V_a\otimes V_b$ 
such that 
\begin{equation}\label{eqn: 3 conditions}
1<a<b, \quad b>4, \quad ab\in2{\Z}. 
\end{equation}
We have $G=H={\SL}_2$, $G_0=H_0=\{\pm1\}$, and 
$\overline{G}=\overline{H}={\PGL}_2$. 
We first 
check the almost-freeness condition $(3)$ in Proposition \ref{general method I}.

\begin{lemma}\label{almost free}
Let $0\leq a<b$ and $b>4$. 
Then ${\PGL}_2$ acts on ${\Gab}$ almost freely. 
\end{lemma}

\begin{proof}
The case $a=0$ is well-known, so we assume $a>0$. 
We first consider the case $b-a\geq4$. 
Observe that for a general point $x\in{\proj}V_b$ and a general $a$-plane $P$ through $x$, 
the orbit ${\PGL}_2\cdot x$ does not intersect with $P$ outside $x$. 
Indeed, if we consider the projection $\pi\colon{\proj}V_b\backslash x\to{\proj}^{b-1}$ from $x$, 
a general $(a-1)$-plane $P'\subset{\proj}^{b-1}$ is disjoint from the $3$-fold $\pi({\PGL}_2\cdot x\backslash x)$. 
Then our claim follows by taking the $a$-plane $P=\overline{\pi^{-1}(P')}$. 
Since $b>4$, $x$ is not fixed by any nontrivial $g\in{\PGL}_2$. 
Then $g$ does not preserve $P$, for otherwise it fixes $x=P\cap({\PGL}_2\cdot x)$. 
This proves the lemma in the range $b-a\geq4$.  
Since we have the dualities 
\begin{equation*}
{\mathbb G}(a, {\proj}V_b) \simeq {\mathbb G}(a, {\proj}V_b^{\vee}) \simeq {\mathbb G}(b-a-1, {\proj}V_b), 
\end{equation*}
the range $a\geq3$ is also covered. 
For the remaining case $(a, b)=(2, 5)$, ${\mathbb G}(2, {\proj}V_5)$ is birationally identified with 
the quotient by ${\PGL}_3$ of the space of morphisms ${\proj}^1\to{\proj}^2$ of degree $5$. 
Since a general rational plane quintic has its six nodes in a general position, 
it has no nontrivial stabilizer in ${\PGL}_3$. 
This derives our assertion for ${\mathbb G}(2, {\proj}V_5)$. 
\end{proof}

We now proceed according to the parity of $b$, assuming \eqref{eqn: 3 conditions}. 

When $b$ is even, the element $-1\in{\SL}_2$ acts on $V_b$ trivially so that 
the bundle $\mathcal{E}$ is already ${\PGL}_2$-linearized. 
Moreover, the quotient ${\proj}V_a^{\oplus a+1}/{\SL}_2$ is rational by Katsylo \cite{Ka3} 
and has dimension $a^2+2a-3>a$. 
Hence the assumptions in Proposition \ref{general method II} are satisfied, and we see that 
\begin{equation*}
{\proj}V_{a,b}/{\SLSL} \sim {\C}^{a(b+1)-3}\times({\proj}V_b/{\SL}_2). 
\end{equation*}
Then ${\proj}V_b/{\SL}_2$ is rational by Katsylo and Bogomolov \cite{Ka2}, \cite{B-K}. 

When $b$ is odd, the element $-1\in{\SL}_2$ acts on $V_b$ by the multiplication by $-1$. 
Hence it acts on $\mathcal{E}$ also by the multiplication by $-1$. 
In this case, since $\mathcal{E}$ has odd rank $a+1$ (remember $ab$ is even), 
$-1\in{\SL}_2$ acts on the Pl\"ucker bundle $\mathcal{L}={\det}\,\mathcal{E}^{\vee}$ by $-1$.  
Then we can twist $\mathcal{E}$ by $\mathcal{L}$ to cancel the action of $-1\in{\SL}_2$. 
Thus the condition $(1)$ in Proposition \ref{general method I} is satisfied. 
As in the case of even $b$, we then deduce that ${\proj}V_{a,b}/{\SLSL}$ is birational to 
${\C}^{a(b+1)-3}\times({\proj}V_b/{\SL}_2)$. 
Now ${\proj}V_b/{\SL}_2$ is rational by Katsylo \cite{Ka1}. 

In this way Theorem \ref{main} is proved for $a>1$, $b>4$.


\section{Rational space curves}\label{sec: space curve}

In the rest of the article we study the cases excluded from \eqref{eqn: 3 conditions} to 
complete the proof of Theorem \ref{main}. 
The cases $(a, b)=(3, 4)$ and $a=1$, $b=2n\geq10$ are settled by Shepherd-Barron 
in \cite{SB1} and \cite{SB2} respectively. 
(In \cite{SB2} he proved the rationality of ${\mathbb G}(1, {\proj}V_b)/{\SL}_2$, 
which by either \eqref{eqn: fibration} or \eqref{eqn:space curve III} is birational to ${\proj}V_{1,b}/{\SLSL}$.) 
Hence the cases to be considered are 
\begin{equation*}
(a, b) = (2, 3), (2, 4), (1, 4), (1, 6), (1, 8). 
\end{equation*}

In this section we study the first three cases by geometric approaches. 
In \S \ref{ssec: space curve} we identify $|{\OQ}(a, b)|$ birationally with 
the space of some parametrized rational space curves for any $(a, b)$. 
Using that description, we study the cases $(a, b)=(2, 3)$ and $(2, 4)$ 
in \S \ref{ssec: (2,3)} and \S \ref{ssec: (2,4)} respectively. 
The case $(a, b)=(1, 4)$ is treated independently in \S \ref{ssec: (1,4)}.

\subsection{Rational space curves}\label{ssec: space curve}

Let $a, b>0$ be any positive integers. 
To a general curve $C$ on $Q={\proj}^1\times{\proj}^1$ of bidegree $(a, b)$ 
we may associate a morphism $\phi_C\colon{\proj}^1\to{\proj}V_b=|{\Oline}(b)|$ 
by regarding $C$ as a family of $b$ points on the second factor ${\proj}^1$ 
parametrized by the first factor ${\proj}^1$. 

\begin{lemma}\label{degree}
The curve $\phi_C({\proj}^1)$ has degree $a$, i.e., $\phi_C^{\ast}{\OPVb}(1)\simeq{\Oline}(a)$. 
\end{lemma}

\begin{proof}
By the Riemann-Hurwitz formula the first projection $C\to{\proj}^1$ has 
$r=2g_C-2+2b$ branch points where $g_C$ is the genus of $C$. 
Substituting $g_C=(a-1)(b-1)$, we have $r=2a(b-1)$. 
These branch points on ${\proj}^1$ correspond to the intersection of $\phi_C({\proj}^1)$ with 
the discriminant hypersurface $D$ in ${\proj}V_b$. 
Since $D$ has degree $2(b-1)$, $\phi_C({\proj}^1)$ has degree $a$. 
\end{proof}

Conversely, given a general morphism $\phi\colon{\proj}^1\to{\proj}V_b$ of degree $a$, 
we obtain a curve on ${\proj}^1\times{\proj}^1$ by pulling back the universal divisor on ${\proj}V_b\times{\proj}^1$. 
Reversing the above calculation, we see that $C$ has bidegree $(a, b)$. 

Let $U_{a,b}$ be the space of morphisms ${\proj}^1\to{\proj}V_b$ of degree $a$, 
on which ${\PGLPGL}$ acts as follows: 
the first factor ${\PGL}_2$ acts on the source ${\proj}^1$ of the morphisms, 
and the second factor ${\PGL}_2$ acts on the target ${\proj}V_b$ in the natural way. 
Then the above construction gives a ${\PGLPGL}$-equivariant birational map 
\begin{equation}\label{eqn:space curve I}
{\proj}V_{a,b}=|{\OQ}(a, b)| \dashrightarrow U_{a,b}. 
\end{equation}
We obtain in particular that 
\begin{equation*}\label{eqn:space curve II}
{\proj}V_{a,b}/{\PGLPGL} \sim U_{a,b}/{\PGLPGL}. 
\end{equation*}
If we denote by $R_{a,b}$ the space of rational curves of degree $a$ in ${\proj}V_b$, 
this may also be written as 
\begin{equation}\label{eqn:space curve III}
{\proj}V_{a,b}/{\PGLPGL} \sim R_{a,b}/{\PGL}_2,  
\end{equation}
where ${\PGL}_2$ acts on $R_{a,b}$ by its action on ${\proj}V_b$. 
Since ${\PGL}_2$ as the subgroup of ${\aut}({\proj}V_b)\simeq{\PGL}_{b+1}$ is 
the stabilizer of a rational normal curve, we have 
\begin{equation*}\label{eqn:space curve IV}
{\proj}V_{a,b}/{\PGLPGL} \sim (R_{a,b}\times R_{b,b})/{\PGL}_{b+1}.  
\end{equation*}
Exchanging $a$ and $b$, we also obtain 
\begin{equation}\label{eqn:space curve V}
{\proj}V_{a,b}/{\PGLPGL} \sim R_{b,a}/{\PGL}_2 \sim (R_{b,a}\times R_{a,a})/{\PGL}_{a+1}. 
\end{equation}

\begin{remark}
The above \eqref{eqn:space curve I} and the description 
${\proj}V_{a,b}\sim{\proj}(V_a\otimes\mathcal{E})$ in \S \ref{sec: main} are connected 
by considering the linear span of $\phi_C({\proj}^1)$, 
which is generically $a$-dimensional and in which $\phi_C({\proj}^1)$ is a rational normal curve. 
\end{remark}

\subsection{The case $(a, b)=(2, 3)$}\label{ssec: (2,3)}

By \eqref{eqn:space curve V} it suffices to prove that $R_{3,2}/{\PGL}_2$ is rational, 
where $R_{3,2}\subset|{\Oplane}(3)|$ is the space of rational plane cubics 
and ${\PGL}_2\subset{\PGL}_3$ is the stabilizer of some reference smooth conic $\Gamma$. 
We may take the homogeneous coordinates $[X, Y, Z]$ of ${\proj}^2$ and 
normalize $\Gamma$ to be defined by $XZ=Y^2$. 

Every rational plane cubic has a unique singularity. 
We apply the slice method for the nodal map 
\begin{equation*}\label{eqn: nodal map cubic}
\kappa : R_{3,2} \to {\proj}^2, \qquad C\mapsto {\rm Sing}C,  
\end{equation*}
which is clearly ${\PGL}_2$-equivariant. 
The group ${\PGL}_2$ acts on ${\proj}^2-\Gamma$ transitively, 
and the stabilizer $G$ of the point $p=[0, 1, 0]$ is isomorphic to $({\Z}/2{\Z})\ltimes{\C}^{\times}$ 
where ${\Z}/2{\Z}$ acts by $[X, Y, Z]\mapsto[Z, Y, X]$ and 
$\alpha\in{\C}^{\times}$ acts by $[X, Y, Z]\mapsto[\alpha^{-1}X, Y, \alpha Z]$. 
The fiber $\kappa^{-1}(p)$ is an open set of the linear system ${\proj}V\subset|{\Oplane}(3)|$ of cubics singular at $p$. 
Hence we have 
\begin{equation*}
R_{3,2}/{\PGL}_2 \sim {\proj}V/G. 
\end{equation*}
 
The group $G$ acts linearly on $V$ and we have the following $G$-decomposition: 
\begin{equation*}
V = \langle XYZ \rangle \oplus \langle X^2Z, Z^2X\rangle \oplus \langle X^2Y, YZ^2\rangle \oplus 
               \langle X^3, Z^3\rangle. 
\end{equation*}
Let 
$W=\langle X^2Z, Z^2X, X^2Y, YZ^2\rangle$, 
$W^{\perp}=\langle XYZ, X^3, Z^3\rangle$, 
and consider the projection 
$\pi\colon{\proj}V\dashrightarrow{\proj}W$ from $W^{\perp}$. 
Then $\pi$ is a $G$-linearized vector bundle. 
Since $G$ acts on ${\proj}W$ almost freely, by the no-name method we have 
\begin{equation*}
{\proj}V/G \sim {\C}^3\times({\proj}W/G). 
\end{equation*}
The quotient ${\proj}W/G$ is rational because it is $2$-dimensional. 
This proves that ${\proj}V_{2,3}/{\PGLPGL}$ is rational.

\subsection{The case $(a, b)=(2, 4)$}\label{ssec: (2,4)}

By \eqref{eqn:space curve V} 
it is sufficient to show that $R_{4,2}/{\PGL}_2$ is rational, 
where ${\PGL}_2$ is the stabilizer in ${\PGL}_3$ of some smooth conic. 
General rational plane quartics have three nodes. 
Let $S^3{\proj}^2$ be the third symmetric product of ${\proj}^2$, and consider the nodal map 
\begin{equation}\label{eqn: nodal map quartic}
\kappa : R_{4,2} \dashrightarrow S^3{\proj}^2, \qquad C\mapsto {\rm Sing}C.  
\end{equation}
General $\kappa$-fibers are open sets of sub-linear systems of $|{\Oplane}(4)|$. 
Since ${\PGL}_2$ acts linearly on $H^0({\Oplane}(4))$, 
$\kappa$ is birationally the projectivization of a ${\PGL}_2$-linearized vector bundle. 
Since ${\PGL}_2$ acts on $S^3{\proj}^2$ almost freely, 
by the no-name lemma we have
\begin{equation*}
R_{4,2}/{\PGL}_2 \sim {\proj}^5\times(S^3{\proj}^2/{\PGL}_2). 
\end{equation*}

Using the slice method (in the converse direction), we see that 
\begin{equation*}
S^3{\proj}^2/{\PGL}_2 \sim (S^3{\proj}^2\times|{\Oplane}(2)|)/{\PGL}_3. 
\end{equation*}
We then apply the slice method to the projection $S^3{\proj}^2\times|{\Oplane}(2)|\to S^3{\proj}^2$. 
The group ${\GL}_3$ acts on $S^3{\proj}^2$ almost transitively, 
and the stabilizer $G$ of 
\begin{equation*}
p_1+p_2+p_3 = [1, 0, 0]+[0, 1, 0]+[0, 0, 1]
\end{equation*} 
is isomorphic to $\frak{S}_3\ltimes({\C}^{\times})^3$ 
where $\frak{S}_3$ acts by the permutations of $X, Y, Z$ and $({\C}^{\times})^3$ is the torus of diagonal matrices. 
Then we have 
\begin{equation*}
(S^3{\proj}^2\times|{\Oplane}(2)|)/{\PGL}_3 \sim |{\Oplane}(2)|/G \sim H^0({\Oplane}(2))/G. 
\end{equation*}

The $G$-representation $H^0({\Oplane}(2))$ is decomposed as 
\begin{equation*}
H^0({\Oplane}(2)) = \langle X^2, Y^2, Z^2\rangle \oplus \langle XY, YZ, ZX\rangle. 
\end{equation*}
We set $W=\langle X^2, Y^2, Z^2\rangle$ and $W^{\perp}=\langle XY, YZ, ZX\rangle$. 
The group $G$ acts on $W$ almost transitively, so that we may apply the slice method to 
the projection $H^0({\Oplane}(2))\to W$ from $W^{\perp}$. 
Hence for the stabilizer $H\subset G$ of a general point of $W$ we have 
\begin{equation*}
H^0({\Oplane}(2))/G \sim W^{\perp}/H. 
\end{equation*}
Then $W^{\perp}/H$ is birational to ${\C}^{\times}\times({\proj}W^{\perp}/H)$, 
and ${\proj}W^{\perp}/H$ is rational because it is $2$-dimensional. 
This completes the proof that ${\proj}V_{2,4}/{\PGLPGL}$ is rational.

\subsection{The case $(a, b)=(1, 4)$}\label{ssec: (1,4)}

The quotient ${\proj}V_{1,4}/{\PGLPGL}$ is birational to 
${\mathbb G}(1, {\proj}V_4)/{\PGL}_2$ by \eqref{eqn:space curve III}. 
Since $V_4\simeq V_4^{\vee}$ as ${\SL}_2$-representations, 
we have a ${\PGL}_2$-equivariant isomorphism 
${\mathbb G}(1, {\proj}V_4)\simeq{\mathbb G}(1, {\proj}V_4^{\vee})$. 
By projecting the standard rational normal curve in ${\proj}V_4^{\vee}$ from lines, 
we obtain a birational map 
\begin{equation*}
{\mathbb G}(1, {\proj}V_4^{\vee})/{\PGL}_2 \dashrightarrow R_{4,2}/{\PGL}_3. 
\end{equation*}
Thus the problem is reduced to the rationality of $R_{4,2}/{\PGL}_3$. 

We apply the slice method to the nodal map \eqref{eqn: nodal map quartic}, 
which we now regard as a ${\GL}_3$-equivariant map. 
We reuse the notations $p_1+p_2+p_3$, $G$ from \S \ref{ssec: (2,4)}. 
Then for the linear system ${\proj}V$ of quartics singular at $p_1+p_2+p_3$ we have  
\begin{equation*}
R_{4,2}/{\PGL}_3 \sim {\proj}V/G \sim V/G. 
\end{equation*}

In terms of the coordinate $[X, Y, Z]$ the $G$-representation $V$ is decomposed as 
\begin{equation*}
V = \langle X^2Y^2, Y^2Z^2, Z^2X^2\rangle \oplus \langle X^2YZ, Y^2ZX, Z^2XY\rangle. 
\end{equation*}
The rest of the proof is similar to the final step in \S \ref{ssec: (2,4)}: 
we may use the slice method for the projection of $V$ from either irreducible summand, 
and then resort to Castelnuovo's theorem to see that $V/G$ is rational. 
Thus ${\proj}V_{1,4}/{\PGLPGL}$ is rational.


\section{Transvectant}\label{sec: transvectant}

In this section we treat the cases $(a, b)=(1, 6), (1, 8)$. 
We first recall in \S \ref{ssec: transvectant} some basic facts about transvectants for biforms. 
In \S \ref{ssec: (1,6)} and \S \ref{ssec: (1,8)}
we study those cases by applying the method of double fibration (\cite{B-K}) to certain transvectants.

\subsection{Transvectants for biforms}\label{ssec: transvectant}

For two representations $V_{a,b}, V_{a',b'}$ of ${\SLSL}$, their tensor product is 
\begin{equation*}\label{eqn: tensor}
V_{a,b}\otimes V_{a',b'} 
= (V_a\boxtimes V_b)\otimes(V_{a'}\boxtimes V_{b'}) 
= (V_a\otimes V_{a'})\boxtimes(V_b\otimes V_{b'}). 
\end{equation*}
Applying the Clebsch-Gordan decomposition for ${\SL}_2$, 
\begin{equation}\label{eqn: CG SL2}
V_d\otimes V_{d'} = \bigoplus_{r=0}^{d''}V_{d+d'-2r}, \qquad d''={\min}\{d, d'\}, 
\end{equation}
we obtain the irreducible decomposition 
\begin{equation*}\label{eqn: CG}
V_{a,b}\otimes V_{a',b'} = \bigoplus_{r, s}V_{a+a'-2r, b+b'-2s},  
\end{equation*}
where $0\leq r\leq{\min}\{a, a'\}$ and $0\leq s\leq{\min}\{b, b'\}$. 
By this decomposition we have an ${\SLSL}$-equivariant bilinear map 
\begin{equation*}\label{eqn: transvectant}
T^{(r,s)} : V_{a,b}\times V_{a',b'} \to V_{a+a'-2r, b+b'-2s}, 
\end{equation*}
unique up to scalar multiplication. 
Let $T^{(r)}:V_d\times V_{d'}\to V_{d+d'-2r}$ be the $r$-th \textit{transvectant}, i.e., 
an ${\SL}_2$-bilinear map associated to \eqref{eqn: CG SL2}. 
Then a standard argument in linear algebra shows that $T^{(r,s)}$ is given (up to constant) by 
\begin{equation}\label{eqn: transvectant II}
T^{(r,s)}(P_1\otimes P_2, \; P_1'\otimes P_2') = T^{(r)}(P_1, P_1')\otimes T^{(s)}(P_2, P_2'), 
\end{equation}
where 
$P_1\in V_{a,0}=V_a$, $P_2\in V_{0,b}=V_b$, $P_1'\in V_{a',0}=V_{a'}$, and $P_2'\in V_{0,b'}=V_{b'}$.  

Let $[X, Y]$ be the homogeneous coordinate of ${\proj}^1$. 
The transvectant $T^{(r)}$ is given explicitly by the following (cf. \cite{Ol}):
\begin{equation}\label{eqn: calc transvectant SL2}
T^{(r)}(P, P') = \sum_{i=0}^{r} (-1)^{i} \binom{r}{i}
                         \frac{\partial^rP}{\partial X^{r-i}\partial Y^i}\frac{\partial^rP'}{\partial X^{i}\partial Y^{r-i}}. 
\end{equation}   
When $r=d'\leq d$ in particular, 
$T^{(d')}(P, P')$ is called the \textit{apolar covariant} and calculated by 
substituting $-\frac{\partial}{\partial Y}$, $\frac{\partial}{\partial X}$ respectively into $X, Y$ in $P'$, 
applying that differential polynomial to $P$, and then multiplying it by $d'!$. 

From \eqref{eqn: transvectant II} and \eqref{eqn: calc transvectant SL2} 
we may calculate the $(r, s)$-th transvectant $T^{(r,s)}$ explicitly 
in terms of the bi-homogeneous coordinate $([X_1, Y_1], [X_2, Y_2])$ of ${\proj}^1\times{\proj}^1$. 
For example, when $a=a'=1$ and $b\geq b'$, we have 
\begin{equation}\label{eqn: transvectant aa'1}
T^{(1,s)}(X_1\otimes P+Y_1\otimes Q, \; X_1\otimes P'+Y_1\otimes Q') = 
T^{(s)}(P, Q') - T^{(s)}(Q, P'), 
\end{equation}
where $s\leq b'$, $P, Q\in V_{0,b}=V_b$, and $P', Q'\in V_{0,b'}=V_{b'}$.

\subsection{The case $(a, b)=(1, 6)$}\label{ssec: (1,6)}

We shall apply the method of double fibration (\cite{B-K}) to the bi-apolar covariant  
\begin{equation*}\label{eqn: transvectant (1,6)}
T^{(1,2)} : V_{1,6}\times V_{1,2} \to V_{0,4}. 
\end{equation*}
Note that ${\dim}V_{1,2}={\dim}V_{0,4}+1$. 
The image of $V_{1,6}\to{\hom}(V_{1,2}, V_{0,4})$ 
given by $H\mapsto T^{(1,2)}(H, \bullet)$ is not contained in the degeneracy locus: 
for example, take $H$ to be $X_1X_2^3Y_2^3+Y_1(X_2^4Y_2^2+X_2^2Y_2^4)$. 
Thus the ${\PGLPGL}$-equivariant map 
\begin{equation*}\label{eqn: double bundle (1,6)}
\varphi : {\proj}V_{1,6} \dashrightarrow {\proj}V_{1,2}, \quad {\C}H\mapsto{\ker}(T^{(1,2)}(H, \bullet)), 
\end{equation*}
is well-defined. 
Note in passing that the $\varphi$-image of the above $X_1X_2^3Y_2^3+Y_1(X_2^4Y_2^2+X_2^2Y_2^4)$ 
defines a smooth curve on ${\proj}^1\times{\proj}^1$. 

\begin{lemma}\label{transitive (1,6)}
The group ${\PGLPGL}$ acts transitively on the open locus $U$ in ${\proj}V_{1,2}$ of smooth curves. 
If we take $C\in U$ to be $X_1Y_2^2+Y_1X_2^2=0$, 
its stabilizer $G$ is isomorphic to $({\Z}/2{\Z})\ltimes{\C}^{\times}$ where 
${\Z}/2{\Z}$ acts by $[X_i, Y_i]\mapsto[Y_i, X_i]$ and 
$\alpha\in{\C}^{\times}$ acts by $[X_1, Y_1]\mapsto[X_1, \alpha^2Y_1]$, $[X_2, Y_2]\mapsto[X_2, \alpha Y_2]$. 
\end{lemma} 

\begin{proof}
By the birational map \eqref{eqn:space curve I} 
$U$ is mapped isomorphically to the space of linear embeddings $\phi\colon{\proj}^1\to{\proj}V_2$ such that 
$\phi({\proj}^1)$ is transverse to the diagonal conic $\Gamma\subset{\proj}V_2$. 
The first assertion holds because the lines in ${\proj}V_2$ transverse to $\Gamma$ are all ${\PGL}_2$-equivalent. 
The stabilizer in ${\PGLPGL}$ of any $C\in U$ is mapped injectively by the projection to the second ${\PGL}_2$, 
and its image is the stabilizer of the pencil $\phi_C({\proj}^1)$. 
Our second assertion follows from this observation and little calculation. 
\end{proof}

By this lemma we may apply the slice method to $\varphi$. 
The $\varphi$-fiber over ${\C}(X_1Y_2^2+Y_1X_2^2)$ is an open set of the projectivization of the linear space 
\begin{equation*}
V = \{ H\in V_{1,6}, \: \, T^{(1,2)}(H, \: X_1Y_2^2+Y_1X_2^2)=0\}. 
\end{equation*}
Then we have 
\begin{equation*}
{\proj}V_{1,6}/{\PGLPGL} \sim {\proj}V/G, 
\end{equation*}
where $G$ is as described in the above lemma. 
The $G$-action on ${\proj}V$ is induced from the linear $G$-action on $V$ given by 
\begin{equation*}
\alpha\in{\C}^{\times} : P_1(X_1, Y_1)P_2(X_2, Y_2) \mapsto 
                                    \alpha^{-4}P_1(X_1, \alpha^2Y_1)P_2(X_2, \alpha Y_2), 
\end{equation*}
where $P_1\in V_{1,0}$ and $P_2\in V_{0,6}$. 

We express elements of $V_{1,6}$ as 
$X_1P+Y_1Q$, 
$P=\sum_{i=0}^6\binom{6}{i}\alpha_iX_2^iY_2^{6-i}$, and 
$Q=\sum_{i=0}^6\binom{6}{i}\beta_iX_2^iY_2^{6-i}$. 
By direct calculation using \eqref{eqn: transvectant aa'1} and \eqref{eqn: calc transvectant SL2}, 
we see that $V$ is defined by 
\begin{equation*}
\alpha_i=\beta_{i+2}, \quad 0\leq i\leq4. 
\end{equation*}
Then we have the $G$-decomposition $V=\oplus_{i=0}^{4}W_i$, where  
\begin{equation*}
W_0 = \langle X_1X_2^2Y_2^4+Y_1X_2^4Y_2^2 \rangle, 
\end{equation*} 
\begin{equation*}    
W_1 = \langle 10X_1X_2^3Y_2^3+3Y_1X_2^5Y_2, \: 3X_1X_2Y_2^5+10Y_1X_2^3Y_2^3\rangle, 
\end{equation*} 
\begin{equation*}    
W_2 = \langle 15X_1X_2^4Y_2^2+Y_1X_2^6, \: X_1Y_2^6+15Y_1X_2^2Y_2^4 \rangle, 
\end{equation*} 
\begin{equation*}    
W_3 = \langle X_1X_2^5Y_2, \: Y_1X_2Y_2^5 \rangle, 
\end{equation*} 
\begin{equation*}    
W_4 = \langle X_1X_2^6, \: Y_1Y_2^6 \rangle. 
\end{equation*}
For $i\geq1$ the $i$-th summand $W_i$ is the induced representation of 
the weight $i$ scalar representation of ${\C}^{\times}$. 
The group $G$ acts almost freely on ${\proj}(W_1\oplus W_2)$. 
Therefore we may apply the no-name method to 
the projection ${\proj}V\dashrightarrow{\proj}(W_1\oplus W_2)$ from $W_0\oplus W_3\oplus W_4$ 
to see that 
\begin{equation*}
{\proj}V/G \sim {\C}^5\times({\proj}(W_1\oplus W_2)/G). 
\end{equation*}
Then ${\proj}(W_1\oplus W_2)/G$ is $2$-dimensional and hence is rational. 
This finishes the proof that ${\proj}V_{1,6}/{\PGLPGL}$ is rational.

\subsection{The case $(a, b)=(1, 8)$}\label{ssec: (1,8)}

We want to show that the $(1, 2)$-th transvectant 
\begin{equation*}\label{eqn: transvectant (1,8)}
T^{(1,2)} : V_{1,8}\times V_{1,4} \to V_{0,8} 
\end{equation*}
determines a double fibration (\cite{B-K}). 
Note that ${\dim}V_{1,4}={\dim}V_{0,8}+1$. 
The non-degeneracy condition is checked, e.g., by the following. 

\begin{lemma}\label{non-degenerate (1,8)}
Take $H=X_1X_2^2Y_2^6+Y_1X_2^6Y_2^2\in V_{1,8}$ and $H'=X_1Y_2^4+Y_1X_2^4\in V_{1,4}$. 
Then we have $T^{(1,2)}(H, H')=0$, and the linear maps 
$T^{(1,2)}(H, \bullet)\colon V_{1,4}\to V_{0,8}$ and $T^{(1,2)}(\bullet, H')\colon V_{1,8}\to V_{0,8}$ 
are both surjective. 
\end{lemma}

\begin{proof}
This is verified by a straightforward (but lengthy) calculation using 
\eqref{eqn: transvectant aa'1} and \eqref{eqn: calc transvectant SL2}. 
We leave it to the reader. 
\end{proof}

Therefore by \cite{B-K} the ${\PGLPGL}$-equivariant map 
\begin{equation*}
{\proj}V_{1,8}\dashrightarrow{\proj}V_{1,4}, \quad {\C}H\mapsto{\ker}(T^{(1,2)}(H, \bullet)), 
\end{equation*}
is well-defined, dominant, and birationally a projective space bundle. 
Explicitly, let 
\begin{equation*}
\mathcal{H} = \{ (H, {\C}H')\in V_{1,8}\times{\proj}V_{1,4}, \; T^{(1,2)}(H, H')=0 \}. 
\end{equation*}
Then $\mathcal{H}$ is generically a sub-vector bundle of $V_{1,8}\times{\proj}V_{1,4}$ 
invariant under the ${\SLSL}$-linearization. 
By the lemma $\mathcal{H}$ has generically the expected rank $9$, 
and the restriction of the natural projection ${\proj}\mathcal{H}\to{\proj}V_{1,8}$ to the main component of ${\proj}\mathcal{H}$ is birational. 
Since ${\SL}_2\times{\PGL}_2$ acts linearly on $V_{1,8}$, $\mathcal{H}$ is in fact ${\SL}_2\times{\PGL}_2$-linearized. 
On the other hand, consider the natural hyperplane bundle ${\sheaf}_{{\proj}V_{1,4}}(1)$ on ${\proj}V_{1,4}$. 
The element $(-1, 1)\in{\SL}_2\times{\PGL}_2$ acts on ${\sheaf}_{{\proj}V_{1,4}}(1)$ by $-1$, 
so that the bundle $\mathcal{H}'=\mathcal{H}\otimes{\sheaf}_{{\proj}V_{1,4}}(1)$ is ${\PGLPGL}$-linearized. 
Then ${\proj}\mathcal{H}'$ is canonically isomorphic to ${\proj}\mathcal{H}$. 
The group ${\PGLPGL}$ acts almost freely on ${\proj}V_{1,4}$, 
for a general rational plane quartic has no nontrivial stabilizer in ${\PGL}_3$ (cf. \S \ref{ssec: (1,4)}). 
Hence we may apply the no-name lemma to $\mathcal{H}'$ to see that 
\begin{equation*}
{\proj}\mathcal{H}/{\PGLPGL} \sim {\proj}\mathcal{H}'/{\PGLPGL} 
\sim {\proj}^8\times({\proj}V_{1,4}/{\PGLPGL}). 
\end{equation*}
In \S \ref{ssec: (1,4)} we proved that ${\proj}V_{1,4}/{\PGLPGL}$ is rational. 
Therefore ${\proj}V_{1,8}/{\PGLPGL}$ is rational.


\end{document}